\documentclass[a4paper,12pt,reqno]{amsart}

\usepackage[english]{babel}
\usepackage[alphabetic]{amsrefs}
\usepackage{amsthm,amssymb}
\usepackage{color}

\usepackage{etoolbox}
\apptocmd{\sloppy}{\hbadness 10000\relax}{}{}

\newtheorem{theorem}{Theorem}[section]
\newtheorem{lemma}[theorem]{Lemma}
\newtheorem{proposition}[theorem]{Proposition}
\newtheorem{corollary}[theorem]{Corollary}
\newtheorem{problem}[theorem]{Problem}
\theoremstyle{definition}
\newtheorem{definition}[theorem]{Definition}

\newtheorem{examples}[theorem]{Examples}
\theoremstyle{remark}
\newtheorem{remark}[theorem]{Remark}

\newcommand{\NN}{\mathbb{N}}
\newcommand{\RR}{\mathbb{R}}
\newcommand{\CC}{\mathbb{C}}

\newcommand{\holo}{\mathcal{O}}
\newcommand{\reg}{\mathcal{O}_{\text{alg}}}
\newcommand{\id}{\mathop{\mathrm{id}}}

\newcommand{\shvecalg}{\mathop{\mathrm{LND}}}

\newcommand{\aut}{\mathop{\mathrm{Aut}}}
\newcommand{\saut}{\mathop{\mathrm{SAut}}}
\newcommand{\authol}{\mathop{\mathrm{Aut}_\mathrm{hol}}}

\newcommand{\autid}{\mathop{\mathrm{Aut}_\mathrm{1}}}

\newcommand{\slgrp}{\mathrm{SL}}
\newcommand{\glgrp}{\mathrm{GL}}

\title{Algebraic Overshear Density Property}

\author{Rafael B. Andrist}
\address{Rafael B. Andrist \\ Department of Mathematics \\
American University of Beirut \\
Beirut, Lebanon. \\ And: Faculty of Mathematics and Physics \\ University of Ljubljana \\ Ljubljana, Slovenia}

\author{Frank Kutzschebauch}
\address{Frank Kutzschebauch \\ Mathematical Institute, University of Bern, Bern, Switzerland}



\begin{document}

\begin{abstract}
We introduce the notion of the algebraic overshear density property which implies both the algebraic notion of flexibility and the holomorphic notion of the density property. We  investigate basic consequences of this stronger property, and propose further research directions in this borderland between affine algebraic geometry and elliptic  holomorphic  geometry.

As an application, we show that any smoothly bordered Riemann surface with finitely many boundary components that is embedded in a complex affine surface with the algebraic overshear density property admits a proper holomorphic embedding.
\end{abstract}

\maketitle

\section{Introduction}
\subsection{History}
Starting with the seminal work of Anders\'en and Lempert \cite{AL} on $\CC^n, \, n \geq 2,$ the study of complex manifolds with
infinite-dimensional holomorphic automorphism groups has been an extremely active area in several complex variables. At the same time, the study of such highly symmetric objects in affine algebraic geometry has been very active as well. In fact, the study of the
algebraic automorphism group of $\CC^n$ has started much earlier than in several complex variables. The starting point of the complex analytic investigations has been motivated by results and questions from the algebraic case. For example, Rosay--Rudin asked \cite{RRmaps}*{Open Question 6, p.~79} in 1988 whether the group of (volume-preserving) holomorphic automorphisms of $\CC^n$ was generated by shears in coordinate directions, i.e., maps of the form
\begin{equation}
\label{shear}
 (z_1, \ldots, z_n) \mapsto (z_1, \ldots , z_{n-1}, z_n + f (z_1, \ldots , z_{n-1})), 
\end{equation}
where $f \in \holo(\CC^{n-1})$ is an arbitrary polynomial or holomorphic function of $n-1$ variables. They can be viewed as time-$1$ maps of the vector field $\theta = f(z_1, \ldots, z_{n-1}) \frac{\partial}{\partial z_n} $.
In complex analysis (when $f$ is holomorphic) such an automorphism is called a shear and such a vector field  is called a shear field. If $f$ is a 
polynomial, the complex analysts call the automorphism a polynomial shear, whereas in affine algebraic geometry it is called an elementary automorphism.  

Overshears in coordinate directions are maps of the form
\begin{equation}
\label{overshear}
(z_1, \ldots, z_n) \mapsto (z_1, \ldots , z_{n-1},  f (z_1, \ldots , z_{n-1}) \cdot z_n), 
\end{equation}
where $f \in \holo^\ast(\CC^{n-1})$ is a nowhere vanishing holomorphic function. By simple connectedness of $\CC^{n-1}$,
the function $f$ is the exponential  $f = e^g$ of some holomorphic $g \in \holo(\CC^{n-1})$. Again, such an automorphism is the 
time-1 map of a complete(ly integrable) holomorphic vector field $\theta = g(z_1, \ldots, z_{n-1}) z_n \frac {\partial} {\partial z_n}$.

The problem in affine algebraic geometry that corresponds to the question of Rosay--Rudin, is the \emph{tame generator conjecture}, asking whether polynomial maps of the form \eqref{shear} together with affine automorphisms (the group generated by them is called the tame subgroup)
generate the algebraic automorphism group of $\CC^n$. This is classically known to be true for $n=2$, much later it has been shown by Umirbaev and Shestakov \cite{US} to be wrong  for $n=3$, and it is still open for $n>3$.
%The result of Umirbaev and Shestakov has led to a change of point of view in affine algebraic geometry.
The notion of tame subgroup which is very much coordinate dependent must be replaced by the group $\saut{\CC^n}$, generated by the flow maps (respectively the corresponding one-parameter subgroups) of locally nilpotent derivations, for short LNDs. The notion of a locally nilpotent derivation is coordinate-independent and makes sense on any affine algebraic variety, see Definition \ref{def-lnd}. The polynomial shears in Equation \eqref{shear} are examples of LNDs in $\CC^n$. Understanding the importance of LNDs in affine algebraic geometry, a group of mathematicians introduced the notion of flexibility which proved extremely useful. The main result of their paper \cite{A-Z}*{Theorem 0.1} states the equivalence of the following three properties for an affine-algebraic manifold $X$:
\begin{enumerate}
\item $X$ is flexible, i.e., the LNDs span the tangent space in each point.
\item The group $\saut(X)$ acts transitively.
\item The group $\saut(X)$ acts infinitely transitively, i.e., $m$-transitive for any $m \in \NN$.
\end{enumerate}
We denote the class of flexible manifolds by $\mathrm{FLEX}$.

Concerning the above-mentioned question of Rosay--Rudin, the answer given by Anders\'en and Lempert is simply no, the group generated by shears and overshears in $\CC^n$ is meagre in the holomorphic automorphism group $\authol{\CC^n}$ (even for $n=2$). 
However, the main result in their paper, the first version of the now so-called Anders\'en--Lempert Theorem, implies that the group generated by shears and overshears in $\CC^n$ is dense in compact-open topology in the holomorphic automorphism group $\authol{\CC^n}$. The Anders\'en--Lempert Theorem, which in current form has been proved by Forstneri\v{c} and Rosay in \cite{ForRos-AL}, led to a number of remarkable geometric constructions in $\CC^n$. We refer the reader to the textbook of Forstneri\v{c} \cite{Forstneric-book}
and to the overview articles of Kaliman--Kutzschebauch \cite{KaKuPresent} and Kutzschebauch \cite{flexi} for an account on this subject. As an interesting example, let us just name the existence of proper holomorphic embeddings of $\varphi \colon \CC^k \hookrightarrow \CC^n$ which are not straightenable, i.e., for no
holomorphic automorphism $\alpha$ of $\CC^n$ its image $\alpha (\varphi (\CC^k))$ is equal to the first $k$-coordinate plane $\CC^k \times \{0\}$. This in turn led to the negative solution of the holomorphic linearization problem. One can construct reductive subgroups of 
$\authol{\CC^n}$, which are not conjugated to a subgroup of linear 
transformations (for details see \cite{DK1}).

The idea behind the Anders\'en--Lempert Theorem was generalized by Varolin to complex manifolds other than $\CC^n$ \cite{density}. He introduced the notion of the density property, see Definition \ref{def-dens}. This is a precise way of saying that the holomorphic automorphism group of a Stein manifold is large.

Methods from algebraic geometry turned out to be very fruitful in the search for manifolds with density property. Already Varolin had introduced the notion of algebraic density property for an affine algebraic manifold, which implies the density property.
However, the algebraic density property is merely a tool for proving the density property, it does, for example, not imply flexibility since it is not using LNDs. For more details on flexibility see \cite{flexi}.

Some of the geometric constructions done in $\CC^n$ with the help of the Anders\'en--Lempert Theorem could be generalized to Stein manifolds with the density property. However, there are constructions which still rely on the coordinates in $\CC^n$. For example, the great embedding results of Riemann surfaces into $\CC^2$ originating in the Ph.D.\ thesis of Wold which can be stated briefly as follows: If a bordered Riemann surface admits a non-proper holomorphic embedding into $\CC^2$, then it also has a proper holomorphic embedding into $\CC^2$ \cite{borderedproper}*{Corollary 1.2}.
The geometric idea behind these results is to push the boundary of the bordered Riemann surface to infinity using a sequence of holomorphic automorphisms. To construct those automorphisms, a combination of the Anders\'en--Lempert Theorem with an explicitly given
shear automorphism is used. The method of ``precomposition with a shear'' goes back to Forstneri\v{c} and Buzzard \cite{BF}*{p.~161}, and has been formalized in the notion of \emph{nice projection property} in \cite{Kutzschebauch-Low-Wold}*{Definition 2.1}. 
All this research described above is part of a newly emerged area of elliptic holomorphic geometry, which also comprises Oka theory, the theory around the Oka--Grauert--Gromov homotopy principle. Stein manifolds with the density property are elliptic in the sense of Gromov and thus they are Oka manifolds. We refer to the textbook of Forstneri\v{c} \cite{Forstneric-book}*{Proposition 5.6.23} for details, see also \cite{flexi}.

\subsection{The new notion}
The aim of this paper is to introduce a new notion of largeness of the holomorphic automorphism group of an affine algebraic manifold. We call it \emph{algebraic overshear density property}. It is stronger than the algebraic density property, in fact it implies both the density property and the notion of flexibility introduced by Arzhantsev et al.\ \cite{A-Z}*{p.~768}. We are confident that this is the correct notion to generalize the geometric constructions known for $\CC^n$ using both the ``nice projection property'' and Anders{\'e}n--Lempert Theory to an affine-algebraic manifold $X$.
Moreover, our notion fits in the modern point of view of affine algebraic geometry concentrating on LNDs and the group $\saut(X)$ generated by their flows. In addition to the powerful Anders\'en--Lempert Theorem we have the theory of locally nilpotent derivations to our disposal, for example, the existence of a quotient  $\pi \colon X \to X//G_a$  which often is affine (e.g. see section \ref{emb})  and is a geometric quotient when restricted to a Zariski open subset \cite{partquot}*{Theorem 4.4}.  This  can be used to replace the use of shears in the ``precomposition with a shear'' trick. 

The paper grew out of a discussion at the conference \emph{Frontiers in Elliptic Holomorphic Geometry} held in Jevnaker which brought together researchers from
affine algebraic geometry and  elliptic holomorphic  geometry. The idea to introduce a new notion closing the gap between different versions of flexibility in the two areas is due to Finnur L\'arusson to whom we express our sincere gratitude. We also thank him for the
careful reading a first manuscript and proposing Problem \ref{problemFLEXAD}.

The paper is organized as follows. In Section \ref{defin}, we recall the 
definitions and define our new property. In the subsequent section we give some geometric consequences of the algebraic overshear density property and a criterion for it. In Section \ref{examples}, we go through the list of affine algebraic manifolds known to have the algebraic density property, and explain which of them also have the algebraic overshear density property. This shows that our new property is strictly stronger than the algebraic density property. In the last section, we propose some open problems.
These mainly concern the geometric constructions which have been developed in $\CC^n$ but not easily generalize to Stein manifolds with the density property.

\section{Definitions}\label{defin}

\begin{definition}
Let $X$ be a complex manifold and let $\Theta$ be a $\CC$-complete holomorphic vector field on $X$. Let $f \in \holo(X)$ such that
\begin{enumerate}
\item $\Theta(f) = 0$, then $f \cdot \Theta$ is called a \emph{shear of $\Theta$} or a \emph{$\Theta$-shear vector field}.
\item $\Theta^2(f) = 0$, then $f \cdot \Theta$ is called an \emph{overshear of $\Theta$} or a \emph{$\Theta$-overshear vector field}.
\end{enumerate}
\end{definition}

\begin{lemma}
Let $X$ be a complex manifold and let $\Theta$ be a $\CC$-complete holomorphic vector field on $X$. Then the shears and overshears of $\Theta$ are also $\CC$-complete holomorphic vector fields on $X$.
\end{lemma}
\begin{proof}
See  \cite{shears}*{Proposition 3.2} for an abstract proof or see  \cite{fibred}*{Lemma 3.3} for an explicit formula of the flow maps.
\end{proof}
\begin{definition}
\label{def-lnd}
Let $X$ be a complex affine variety. A \emph{locally nilpotent derivation} (LND) on $X$ is $\CC$-linear derivation $D \colon \CC[X] \to \CC[X]$ such that for each $f \in \CC[X]$ there exists a number $n \in \NN$ such that $D^n(f) = 0$.

We denote set of the \emph{algebraic shear vector fields} of $X$ by $\shvecalg(X)$. These are precisely the LNDs on $X$.
\end{definition}

\begin{definition}
Let $X$ be a smooth complex affine variety (we call such a variety below a complex affine manifold).
\begin{enumerate}
\item The set of \emph{holomorphic shear vector fields} of $X$ consist of all holomorphic vector fields that are holomorphic shears of vector fields in $\shvecalg(X)$.
\item The set of \emph{algebraic overshear vector fields} of $X$ consist of all algebraic vector fields that are algebraic overshears of vector fields in $\shvecalg(X)$.
\item The set of \emph{holomorphic overshear vector fields} of $X$ consist of all holomorphic vector fields that are holomorphic overshears of vector fields in $\shvecalg(X)$.
\end{enumerate}
\end{definition}

\begin{definition}
Let $X$ be a complex affine manifold.
We say that $X$ enjoys the \emph{algebraic overshear density property} if the Lie algebra generated by the algebraic overshear vector fields coincides with the Lie algebra of all algebraic vector fields.
We denote the class of these manifolds by $\mathrm{AOSD}$.
\end{definition}

As a comparison, we give also the original definition of the (algebraic) density property introduced by Varolin \cite{density}*{Section 3}.

\begin{definition}
\label{def-dens} \hfill
\begin{enumerate}
\item Let $X$ be a complex affine manifold.
We say that $X$ enjoys the \emph{algebraic density property} if the Lie algebra generated by the complete algebraic vector fields coincides with the Lie algebra of all algebraic vector fields.
\item Let $X$ be a complex manifold.
We say that $X$ enjoys the \emph{density property} if the Lie algebra generated by the complete holomorphic vector fields is dense (w.r.t.\ the topology of locally uniform convergence) in the Lie algebra of all holomorphic vector fields.
\end{enumerate}
\end{definition}

%\begin{definition}
%Let $X$ be a complex affine manifold.
%\begin{enumerate}
%\item The \emph{algebraic shear group}, denoted by $\shgrpalg(X)$, is the group generated by flows of algebraic shear vector fields on $X$.
%\item The \emph{algebraic overshear group}, denoted by $\oshgrpalg(X)$, is the group generated by flows of algebraic overshear vector fields on $X$.
%\end{enumerate}
%\end{definition}

\section{Geometric consequences and a criterion}
\begin{proposition}
Let $X$ be a complex affine manifold. If $X$ has the algebraic overshear density property, then it is flexible.
\label{aosd-flex}
\end{proposition}

\begin{proof}
Let $x_0 \in X$ be a point. We will now find finitely many LNDs on $X$ that span the tangent space in $x_0$.

Since $X$ is affine, we can span the tangent space in $x_0$ by finitely many algebraic vector fields $\Theta_1, \dots, \Theta_n$ where $n = \dim X$.

The algebraic overshear density property implies that
for each $j = 1, \dots, n$ there exist finitely many LNDs $\Theta_{j,1}, \dots, \Theta_{j, m(j)}$ and regular functions $f_{j,1} \in \ker \Theta_{j,1}^2, \dots, f_{j,m(j)} \in  \Theta_{j, m(j)}^2$ such that each
$\Theta_j$ is a Lie combination of $f_{j,k} \cdot \Theta_{j,k}$ where $k = 1, \dots, m(j)$.

We make the following observation:
Let $f \cdot \Theta$ resp. $\widetilde{f} \cdot \widetilde{\Theta}$ be overshears of LNDs $\Theta$ resp.\ $\widetilde{\Theta}$. Then we obtain for their Lie bracket:
\begin{equation*}
\left[ f \cdot \Theta, \widetilde{f} \cdot \widetilde{\Theta}  \right] = f \Theta(\widetilde{f}) \cdot \widetilde{\Theta} - \widetilde{f} \widetilde{\Theta}(f) \cdot \Theta +  f \widetilde{f} \left[ \Theta, \widetilde{\Theta} \right]
\end{equation*}
Hence, we see that each vector $\Theta_j | x_0$ is a complex linear combination of the LNDs $\Theta_{j,k}$ and their Lie brackets, evaluated in $x_0$.
Moreover, a Lie bracket of two LNDs $\Theta$ and $\widetilde{\Theta}$ can be expressed using the flow map $\varphi_t$ of $\Theta$:
\begin{equation*}
\left[ \Theta, \widetilde{\Theta} \right]_{x} = \lim_{t \to 0} \frac{ \mathop{d \varphi_{-t}} \widetilde{\Theta}_{\textstyle \varphi_t(x)} - \widetilde{\Theta}_{\textstyle x}}{t}
\end{equation*}
Since $\Theta$ is an LND, its flow $\varphi_t$ is algebraic, and we can approximate $\left[ \Theta, \widetilde{\Theta} \right]$ arbitrarily well by the sum of two LNDs. Applying these arguments inductively, we see that each vector $\Theta_j | x_0$ can be approximated arbitrarily well by a complex linear combination of LNDs. Since spanning the tangent space is an open condition, we conclude that finitely many LNDs span the tangent spaces in an open neighborhood of $x_0$.

The set where these LNDs do not span the tangent spaces, must be an algebraic subvariety. Hence, we obtain the desired conclusion by a standard argument.
\end{proof}

We will need the following notion of a \emph{compatible pair} which was introduced by Kaliman and Kutzschebauch \cite{densitycriteria}*{Definition 2.5}:

\begin{definition}
\label{def-pair}
Let $X$ be an affine algebraic manifold $X$. A \emph{compatible pair of LNDs} are $\Theta, \, \Xi \in \shvecalg(X)$ such that
\begin{enumerate}
\item $\exists \, h \in (\ker \Theta^2 \setminus \ker \Theta) \cap \ker \Xi$
\item $\mathop{\mathrm{span}}_\CC \{ \ker \Theta \cdot \ker \Xi \}$ contains a non-trivial $\reg(X)$-ideal
\end{enumerate}
\end{definition}

\begin{proposition}\label{comppairs}
Let $X$ be an affine algebraic manifold that is flexible and a admits a compatible pair of LNDs. Then $X$ has the algebraic overshear density property.
\end{proposition}

\begin{proof} The proof follows the lines of proof  of Theorems 1 and 2 from \cite{densitycriteria}. The compatible pair creates a submodule
of the algebraic vector fields which is contained in the Lie algebra generated by algebraic overshear vector fields. Then
transitivity of $\saut(X)$ follows from flexibility, and the transitivity
on a tangent space is also an easy consequence of flexibility. In 
fact, in \cite{A-Z}*{Theorem 4.14 and Remark 4.16} even more is proven, namely that any element in $\slgrp (T_p X)$ is contained in the isotropy group of $\saut(X)$ at a point $p \in X$. Since the pull-back of an LND by an algebraic automorphism is again an LND, the above reasoning shows that the Lie algebra generated by algebraic overshear vector fields is equal to the Lie algebra of all algebraic vector fields on $X$.
\end{proof}
\begin{theorem}
Let $X$ be a complex affine manifold with the algebraic overshear density property.
Then the group generated by flows of algebraic overshear vector fields is dense in the identity component $\autid(X)$ in compact-open topology.
\end{theorem}
\begin{proof}
This is a direct application of the Anders\'en--Lempert Theorem, which can be found implicitly in \cites{AL, ForRos-AL, ForRos-AL-err,density}. For $\CC^n, n \geq 2,$ the best explicit reference is \cite{Forstneric-book}*{Theorem 4.9.2} with $\Omega = \CC^n$. For the general case, the analogous theorem can be found explicitly in \cite{KaKuPresent}*{Theorem 2} where one chooses $\Omega = X$ and any $\Phi_1 \in \autid(X)$. Only the flows of those complete vector fields are used for the approximation that were needed to establish the density property. This is implicit in the proof of \cite{KaKuPresent}*{Theorem 2}.
\end{proof}

\begin{theorem}
Let $X$ be a complex affine surface with the algebraic overshear density property and let $Y$ be a complex manifold. If the holomorphic endomorphisms of $X$ and the holomorphic endormophisms of $Y$ are isomorphic as semi-groups, then $X$ and $Y$ are biholomorphic or anti-biholomorphic.
\end{theorem}
\begin{proof}
A general orbit of the $\CC^+$-action generated by an LND gives  an algebraic embedding of $\CC \hookrightarrow X$ which is automatically proper. The theorem then follows from a theorem of Andrist \cite{endos}*{Theorem 3.3}.
\end{proof}
\begin{remark}
For a manifold $X$ with the density property there exists a proper holomorphic immersion $\CC \to X$. If in addition $\dim X \geq 3$, then there exists a proper holomorphic embedding $\CC \to X$ (see Andrist and Wold \cite{embedRiemann}*{Theorem 5.2}) and the above theorem follows. If $\dim X = 2$, we do not know in general, whether a proper holomorphic embedding $\CC \to X$ exists.
\end{remark}

\section{(Algebraic) density versus Algebraic overshear density -- all known examples}
\label{examples}
We will go through the list of known examples of affine-algebraic manifolds with the density property, and indicate for each one of them whether or not they have the algebraic density property resp.\ the algebraic overshear density property. As a general feature one can see that the gap between algebraic density property and algebraic overshear density property as rather narrow, but existent.
\begin{enumerate}
\item 
A homogeneous space $X=G/H$ where $G$ is a linear algebraic group and $H$ is a closed algebraic subgroup such that $X$ is affine and whose connected components are different from $\CC$ and from $(\CC^*)^n, n\geq 1,$ has algebraic density property. 

\medskip
It is known that if $H$ is reductive, then the space $X=G/H$ is always affine. However, there is no known group-theoretic criterion that would characterize when $G/H$ is affine.
The above result has a long history, it  includes all examples known from the work of Anders\'en--Lempert, Varolin \cite{density}, Varolin--Toth \cite{ToVa}, Kaliman--Kutzschebauch \cite{densitycriteria}, and Donzelli--Dvorsky--Kaliman \cite{DoDvKa2010}. The final result has been obtained by Kaliman--Kutzschebauch in \cite{densityhomog}.
The one-di\-men\-sion\-al manifolds $\CC$ and $\CC^*$ do not have the density property, however the following problem is well known and seems notoriously difficult:

\medskip
\textbf{Open Problem:} Does $(\CC^*)^n, n\ge 2,$ have the density property? 

\smallskip
It is conjectured that the answer is no, more precisely one expects that
all holomorphic automorphisms of $(\CC^*)^n, n\ge 2$ preserve the form $\wedge_{i=1}^n \frac{ {\rm d} z_i}{  z_i}$ up to sign.
\end{enumerate}

We are able to characterize those homogeneous spaces $X$ from above which have the algebraic overshear density property.
\begin{theorem}
\label{thm-homog}
Let $X = G/H$ be an affine homogeneous space of a linear algebraic group $G < \glgrp_n(\CC)$ by a closed algebraic subgroup $H$. Then $X$ has the algebraic overshear density property if and only if all algebraic morphisms $X \to \CC^\ast$ are trivial.
\end{theorem}
\begin{proof}
Since the algebraic maps from $\CC$ to $\CC^\star$ are constant, the flow curves of
an LND have to be tangent to the fibers of any algebraic morphism $X \to \CC^\ast$.
This proves the necessity part. To see the sufficiency of the condition,
we follow the proof of the algebraic density property for $X$ of Kaliman and Kutzschebauch in \cite{densityhomog}*{Theorem~11.7}. That proof needs only LNDs and their algebraic overshears if (and only if) the torus $T_1$ introduced in \cite{densityhomog}*{Notation~4.10, p.~1319} is trivial, which in turn follows from the triviality of all algebraic morphisms $X \to \CC^\ast$, and by \cite{densityhomog}*{Proposition~5.2.} $X$ is isomorphic to $Z \times T_1$ where $Z$ is an affine flexible variety by \cite{densityhomog}*{Lemma 4.11}. Thus, by Proposition 3.3 it suffices to establish existence of a compatible pair of LNDs on $Z$. There are two cases, and in the first one such existence follows from \cite{densityhomog}*{Proposition 7.1}. In the second one, there is a construction of compatible pairs in the proof of \cite{densityhomog}*{Corollary 8.2} where the assumptions of the latter corollary are checked for $Z$ in the proof of \cite{densityhomog}*{Theorem 10.6}.
\end{proof}

\begin{examples}[for Theorem \ref{thm-homog}] \hfill
\begin{itemize}
\item The motivating example is of course $\CC^n$, $n \geq 2$, where it is sufficient to take algebraic overshears of $\frac{\partial}{\partial z_1}, \dots, \frac{\partial}{\partial z_n}$ to establish the algebraic overshear density property as it was shown by Anders\'en and Lempert \cite{AL}.
\item $\glgrp_n(\CC), n \geq 2,$ has the algebraic density property, but not the algebraic overshear density property, since $\det \colon \glgrp_n(\CC) \to \CC^\ast$ is non-trivial.
\item $\slgrp_n(\CC), n \geq 2,$ has the algebraic overshear density property.
\end{itemize}
\end{examples}

We continue our list of all known examples for the algebraic overshear density property.

\begin{enumerate}
\setcounter{enumi}{1}
\item
The manifolds  $X$ given as a submanifold in $\CC^{n+2}$ with coordinates $u\in \CC$, $v\in \CC$, $z\in \CC^n$  by the equation $uv = p(z)$, where the zero fiber of the polynomial $p \in \CC[\CC^n]$ is smooth (otherwise $X$ is not smooth)
 have algebraic density property and thus density property \cite{KKhyper}.
We claim that all these examples have algebraic overshear density property:
\begin{enumerate}
\item For $n=1$ this is the main result of \cite{Danielewski}*{Section 3}. Note that for $n=1$ this is also an example of a complex manifold with density property that does not admit compatible pairs (see Definition~\ref{def-pair}) which follows from
Proposition~2.9 in \cite{densityalgvol}.

\item For $n \geq 2$ it is an easy exercise to show that the group $\saut(X)$ acts transitively on $X$ and that the pairs of LNDs $(\theta_i, \theta_j), i \neq j$ form compatible pairs where 
$\theta_i := \frac{\partial p}{ \partial z_i} \frac{\partial}{\partial u} - u \frac{\partial}{\partial z_i}$. The result now follows from Proposition \ref{comppairs}.
\end{enumerate}

\item
\label{ex-giz}
Before formulating the next result, recall that \emph{Gizatullin surfaces} are by definition the normal affine surfaces on which the algebraic automorphism group acts with an open orbit whose complement is a finite set of points. By the classical result of Gizatullin, they can be characterized by admitting a completion with a simple normal crossing chain of rational curves at infinity. Every Gizatullin surface admits a $\CC$-fibration with at most one singular fibre which is however not always reduced.

Smooth Gizatullin surfaces which admit such a $\CC$-fibration that the singular fibre is reduced (sometimes called \emph{generalized Danielewski surfaces}) have the density property \cite{Gizatullindens}. We do not know except for the case of Danielewski surfaces whether 
they have the algebraic overshear density property, since in the proof of the density property one uses LNDs pulled back by certain holomorphic (non-algebraic) automorphisms arising as flow maps of algebraic vector fields. These pullbacks are not necessarily algebraic vector fields.

\item
Certain algebraic hypersurfaces in $\CC^{n+3}$:
\[
\{ (x, y, z_0, z_1, \dots, z_n) \in \CC \times \CC \times \CC^{n+1} \,:\, x^2 y = a(z) + x \cdot b(z) \}
\]
where $\deg_{z_0} a \leq 2$, $\deg_{z_0} b \leq 1$ and not both degrees are zero. This class includes the Koras--Russell cubic threefold and has the density property according to Leuenberger \cite{KorasRussell}.

The Koras--Russell cubic $\{ x + x^2 y + u^2 + v^3 = 0 \} \subset \CC^4_{x,y,u,v}$ which is famous for being diffeomorphic to $\RR^6$, but being not algebraically isomorphic to $\CC^3$ cannot be flexible and, hence, does not have the algebraic overshear density property, since all LNDs necessarily fix the $x$-coordinate axis, see Makar-Limanov \cite{MakarLimanov}. We do not know whether those examples of Leuenberger have algebraic density property by the same reason as in the class of examples \ref{ex-giz} above.

\item The Calogero--Moser spaces $\mathcal{C}_n$ have the algebraic overshear density property. An inspection of the proof in \cite{MR4305975} shows that only overshears are needed: The density property follows from the existence of a non-degenerate algebraic $\mathrm{SL}_2(\CC)$ action and from flexibility. Hence, all the involved vector fields are overshears of LNDs.
\end{enumerate}

\clearpage

Recently, Ugolini and Winkelmann proved that if $X$ is a complex affine manifold that is flexible, and if $E \to X$ is an algebraic vector bundle over $X$, then $E$ is flexible as well \cite{UgoWin}*{Theorem 1.6}. Moreover, they proved that if $X$ is a complex affine manifold with the density property, and if $E$ is algebraic and flexible, then $E$ has the density property  \cite{UgoWin}*{Corollary 1.5}. Thus, if $X$ has the algebraic overshear density property, then it follows from their results that the total space $E$ of an algebraic vector bundle $E \to X$ has the density property. However, their proof requires as an essential ingredient the Euler vector field which is a $\CC^\ast$-action and not an LND. This leads naturally to the  question, whether the total space admits the overshear density property. Indeed, this is true by our next result.

\begin{theorem}
\label{thm:bundles}
Let $X$ be a flexible complex affine manifold. Let $E \to X$ be an algebraic vector bundle over $X$. Then $E$ has the algebraic overshear density property.
\end{theorem}

%We need the following theorem of Rosenlicht (cited after Springer \cite{Springer}*{Proposition 14.2.2}).

%\begin{theorem}
%\label{thm:Rosenlicht}
%Let $(\CC,+)$ act non-trivially on an irreducible affine algebraic variety $X$. Then there exists an affine algebraic variety $Y$ such that
%\begin{enumerate}5
%\item There exists an isomorphism $\Phi \colon \CC \times Y \to \Omega \subset X$ where $\Omega$ is a non-empty Zariski-open subset.
%\item There exists a morphism $\Psi \colon \CC \times Y \to \CC$ with
%\[
%\psi(a + b, y) = \psi(a, y) + \psi(b, y) \text{ and }
%a \cdot \phi(b, y) = \phi(\psi(a, y) + b, y)
%\]
%for all $a, b \in \CC$ and $y \in Y$.
%\end{enumerate}
%\end{theorem}

%\begin{remark}
%For a LND $\Theta \neq 0$ on $X$, this implies that $\Phi^\ast \Theta_{|\Omega} = \frac{\partial}{\partial z}$ where $z \in \CC$ is the coordinate of the first factor. Moreover, $\Omega$ is $\Theta$-invariant.
%\end{remark}
Before we start with the proof let us remind some
general construction for an LND $\Theta$ acting on an affine-algebraic variety $X$.
\begin{remark}
\label{Rosenlicht}
We quote the following from \cite{A-Z}*{Remarks~2.7}: 
According to \cite{partquot}*{Theorem~3.3} the field of rational invariants is the quotient field of the ring $\CC[X]^\Theta$ of regular invariants for an affine-algebraic variety $X$ for an LND $\Theta$.
Hence by a corollary of the Rosenlicht theorem on rational invariants (see \cite{partquot}*{Proposition~3.4}) the regular invariants $\CC[X]^\Theta$ separate orbits on an $\Theta$-invariant open dense subset $U(\Theta)$ of $X$. Shrinking $U(\Theta)$ if necessary we can
achieve that $U(\Theta)$ has a geometric quotient $U(\Theta)/\Theta$, which admits a locally closed embedding into some $\CC^N$ by regular invariants in $\CC[X]^\Theta$ (see also \cite{partquot}*{Theorem 4.4}).
\end{remark}

\begin{proof}[Proof of Theorem \ref{thm:bundles}]
We choose a non-trivial LND $\Theta$ on $X$ which exists due to flexibility. We now apply the Remark \ref{Rosenlicht} above to the $(\CC, +)$-action induced by $\Theta$ and obtain an $\Theta$-invariant open dense subset $U(\Theta)$ of $X$. Also by the same Rosenlicht theorem (see e.g.\ Springer \cite{Springer}*{Proposition 14.2.2}) we have an isomorphism $\Phi \colon \CC_z \times Y \to U(\Theta)$ where $Y$ is an affine-algebraic variety $Y$, and $\Phi^\ast(\Theta) = \frac{\partial}{\partial z}$.

Moreover, the algebraic vector bundle $\pi \colon E \to X$ is trivial over a non-empty Zariski-open subset $V \subset X$.
In particular, the bundle is trivial over $U(\Theta) \cap V$. On $\CC \times Y$, we can extend the trivialization from $\Phi^{-1}(U(\Theta) \cap V)$ to $\CC \times Y'$ for some non-empty Zariski-open subset $Y' \subset Y$ using the action of $\Phi^\ast(\Theta)$ and the fact that an algebraic vector bundle is invariant under the flow of of an LND of the base space (see e.g.\ \cite{UgoWin}*{Theorem 11.1}).

We set $W := \Phi(\CC \times Y')$. 
Since $W \subset X$ is a $\Theta$-invariant Zariski-open subset of the affine variety $U(\Theta)$, there exists a non-zero regular function $f \colon X \to \CC$ with $W \subset \{f \neq 0\}$ that is $\Theta$-invariant: In fact, we can find a non-zero function which is a polynomial in the finitely many regular invariants of $\CC[X]^\Theta$ that embed $Y \cong U(\Theta)/\Theta$ as a locally closed subset into $\CC^N$.
 
Let $r \geq 1$ be the rank of $E$. A large affine part of $E$ is now given by $W \times \CC^r \cong (\CC_z \times Y') \times \CC^r_w$.
By $\hat{\Phi}$ we also denote the extension of the isomorphism $\Phi$ to the trivial bundle, $\hat{\Phi} \colon (\CC_z \times Y') \times \CC^r_w \to W \times \CC^r$. 
Let $\Xi := \hat{\Phi}_\ast \frac{\partial}{\partial w_1}$. Then $\widetilde{\Theta} := (f \circ \pi) \cdot \hat{\Phi}_\ast \Phi^\ast \Theta$ and $\widetilde{\Xi} := (f \circ \pi) \cdot \Xi$ are well defined algebraic vector fields on $E$ and are in fact LNDs. 
On $ (\CC_z \times Y') \times \CC^r_w$, the vector fields $\frac{\partial}{\partial z}$ and $\frac{\partial}{\partial w_1}$ form a compatible pair, since the span of the product of their kernels contains in fact all regular functions. Hence, $\widetilde \Theta$ and $\widetilde \Xi$ form a compatible pair on $E$, since the span of the product of their kernels contains the ideal generated by $f \circ \pi$.

Moreover, we have flexibility of the total space $E$ by \cite{UgoWin}*{Theorem 11.4}, and hence the result follows from Proposition \ref{comppairs}.
\end{proof}

\clearpage

\section{Embedding bordered Riemann surfaces}
\label{emb}

We generalize the following result from $\CC^2$ to complex-affine surfaces with the algebraic overshear density property. 

\begin{theorem}
\label{thmembedding}
Let $X$ be a complex affine surface with the algebraic overshear density property and let $R \subset X$ be an embedded, smoothly  bordered Riemann surface with finitely many boundary components and non-empty boundary.  Then there exists a proper holomorphic embedding of the Riemann surface $R$ into $X$.
\end{theorem}

\begin{remark} The exact assumption in the Theorem is that $R^\circ$ is the interior of a closed domain with boundary of class $\mathcal{C}^r$, $r \geq 2$, in a compact (not necessarily connected) Riemann surface $\hat R$.  We may assume that the boundary consists of real analytic Jordan curves, and $R$ has no connected component without boundary, see Stout \cite{Stout}*{Theorem 8.1}. Note that any conformal diffeomorphism of $R$ onto such a domain is in $\mathcal{C}^1(\hat{R})$; see \cite{AFL-book}*{Theorem 1.10.10}. By \emph{embeddded}, we mean that there exists an injective and immersive $\mathcal{C}^1$-map $f \colon R \to X$ that is 
holomorphic in the interior $R^\circ$ of $R$. By an application of the Mergelyan theorem as in
the proof of Corollary 1.2 in \cite{borderedproper} we may assume that $f$ is in fact holomorphic in an open neighborhood $U$ of $R$ in $\hat{R}$. The corresponding form of the manifold-valued Mergelyan Theorem can be found in \cite{legacy}*{Corollary 8} on p.~178.
\end{remark}

\begin{remark}
\label{trivialGizatullin}
A complex-affine surface $X$ with the algebraic overshear density property has trivial Makar-Limanov invariant, since it is flexible (\cite{A-Z}*{Proposition 5.1}). For a normal complex-affine surface non-isomorphic to $\CC \times \CC^\ast$ and $\CC^\ast \times \CC^\ast$, this is one of the equivalent conditions implying that $X$ is a Gizatullin surface (see Bandman and Makar-Limanov \cite{MR1872757} in the smooth case, and Dubouloz \cite{MR2069802} in the normal case). Together, this implies that our $X$ above is a homogeneous Gizatullin surface. We will however not make use of this fact in the following proof.
We do not know whether the result of Theorem \ref{thmembedding} extends to all homogeneous Gizatullin surfaces. Our proof uses the density property which is not established for all homogeneous Gizatullin surfaces. The first author proved that Gizatullin surfaces with reduced degenerate fibre have the density property \cite{Gizatullindens} which does not cover all homogeneous Gizatullin surfaces.
\end{remark}

Among the ingredients, we will need Zariski's finiteness theorem:
\begin{theorem}[\cite{Zariski}, cited after \cite{Freudenburg}*{Section~6.3}]
For a field $k$, let $A$ be an affine normal $k$-domain, and let $K$ be a subfield of $\mathrm{frac}(A)$ containing $k$.
If $\mathrm{tr.deg.}_k K \leq 2$, then $K \cap A$ is finitely generated over $k$.
\end{theorem}

\begin{corollary}
Let $X$ be a complex-affine surface, and let $\Theta \in \shvecalg{X}$. Then the GIT quotient $X /\!/ \CC^+_\Theta$ is affine.
\end{corollary}

Any non-trivial LND $\Theta$ admits a \emph{local slice}, i.e., in our context a polynomial function $f \colon X \to \CC$, such that $\Theta(f) \neq 0$ but $\Theta^2(f) = 0$.
If the GIT quotient $X /\!/ \CC^+_\Theta$ is affine, then there exists a Zariski-open subset $U$ in the quotient s.t.\ the corresponding fibration over $U$ is trivial and each such fibre is isomorphic to $\CC$, see \cite{Freudenburg}*{Principle 11, p.~27} and the comment \cite{Freudenburg}*{p.~34}.

\bigskip

The proof of the main theorem of this section uses the idea of a ``pre-composition with a shear'' from the original proof of Forstneri\v{c} and Wold in the case of $X = \CC^2$ \cite{borderedproper}. It has been generalized to $X = \CC \times \CC^\ast$ by Ritter and L{\'a}russon \cite{MR3233212} and to $X = (\CC^\ast)^2$ by Ritter \cite{MR3881472}. We give a short outline of their proof and indicate the necessary adjustments to our situation.

\begin{proof}
According to Proposition \ref{aosd-flex}, we can choose two LNDs $\Theta_1$ and $\Theta_2$ on $X$ that are not proportional.
Denote by $\pi_k \colon X \to X /\!/ \CC^+_{\Theta_k}$ the projection to the GIT quotient of $\Theta_k$ for $k=1,2$. Note the following facts:
\begin{enumerate}
\item $X /\!/ \CC^+$ is a smooth affine curve, since $X$ is a smooth surface, and thus  the quotient is normal which excludes singularities. In addition, the existence of a second non-proportional locally nilpotent derivation gives
a non-constant map of $\CC$ (any generic orbit of this second LND) into
$X /\!/ \CC^+$ and thus $X /\!/ \CC^+$ is isomorphic to $\CC$.

\item The general fibres of $\pi_k$ are isomorphic to $\CC$ with only finitely many special fibres.
\end{enumerate}

In the first step of the proof, we expose points w.r.t.\ $\pi_2$.
We indicate the necessary adjustments to the method introduced by Forstneri\v{c} and Wold \cite{borderedproper}*{Section 4}.

\begin{definition}
A point $p \in \partial R$ is called \emph{exposed} w.r.t.\ $\pi_k \; (k = 1, 2)$ if 
\[
\pi_k^{-1} (\pi_k (p)) \cap \overline{R} = \{p\}
\]
Moreover, $F_p:=\pi_k^{-1} (\pi_k (p))$ and the boundary $\partial R$ intersect transversely, in fact $T_p F_p \cap T_p \partial R =\{0\}$.
\end{definition}

Following the procedure in \cite{borderedproper}*{Section 4} we change the embedding of $R$ into $X$ so that in each of the finitely many boundary curves $(\partial R)_i$ there exists an exposed point $p_i$
and such that each of the exposed points lies in a general fibre of $\pi_2$. The proof goes exactly as the proof of Theorem 4.2 in \cite{borderedproper}. The conclusion that $p_i$ is in a general fibre can be easily achieved by adding this assumption to choice of the points $p_i$ (same notation in the proof of Theorem 4.2) before choosing the arcs $\lambda_i$ in that proof (see also Figure 2 in \cite{borderedproper}). The Mergelyan Theorem for maps to  $\CC^2$ (functions)
should be replaced by the (Oka) manifold-valued Mergelyan Theorem proved by Forstneri\v{c} in \cite{Mergelyan} which can be found in \cite{legacy}*{Corollary 8}. 
By the definition of an exposed point we then have  $\pi_2(p_i) \neq \pi_2(p_j)$ for $p_i \neq p_j$.

%Indeed, since $X$ has the density property, we can use Anders\'en--Lempert theory to achieve the above properties by proceeding as in the case of $\CC^2$. 

%The projections $\pi_2(p_i)$ will be contained in some disk $\Delta$.

In the second step, we follow the procedure in \cite{borderedproper}*{Section 5} which in turn relies on Wold \cite{woldembed}*{Section 4}, in particular Lemma 1 therein. We consider $\pi_1$ restricted to $\pi_2^{-1}\{ \mathrm{q} \}$ for a point $q \in X /\!/ \CC^+_{\Theta_2} \cong \CC$ where in fact we will choose $q = \pi_2(p_i)$. Since this fibre is isomorphic to $\CC$ as well, we can find algebraic coordinates and write $\pi_1 | \pi_2^{-1}\{ \mathrm{q} \}$ as a polynomial map $\CC \to \CC$. Choosing a sufficiently small Euclidean neighborhood $D \subset \CC$ around $q$ such that all fibers above points in $D$ are general, we can write 
\[
\pi_1(z,w) = a_n(w) \cdot z^n + a_{n-1}(w) \cdot z^{n-1} + \dots + a_1(w) \cdot z + a_0(w)
\]
where $w \in D$ and $z \in \CC$ is in the fibre above $w$ w.r.t.\ $\pi_2(z,w) = w$. The coefficients $a_i \in \holo (D)$ are holomorphic functions.
Note that this map can't be of degree $0$ in $z$ since the LNDs $\Theta_1$, $\Theta_2$ are not proportional.

Let $\varphi_2(x,t)$ with time $t \in \CC$ and $x \in X$ be the (complete) flow of $\Theta_2$.

Consider the map
\[
x \mapsto \Phi(x) := \varphi_2 \left(x, \sum_{i} \frac{\alpha_i}{\pi_2(x) - \pi_2(p_i)} \right) 
\]
with $\alpha_i \in \CC \setminus \{0\}$.
When restricted to $R$, this sends precisely the exposed points to infinity. The choices of the finitely many constants $\alpha_i$  have to be made such that the $\pi_1$-projections of the paths given by ``the opened boundary components'' $u_i \colon \RR \to \Phi ((\partial R)_i) $ are pairwise disjoint outside a compact. Note that these projections lie in the quotient isomorphic to $\CC$. Hence, we can proceed with the choice of constants $\alpha_i$ as in the original proof for $\CC^2$ where the projection was onto a coordinate axis; we just need to notice the following:

Without loss of generality, we may assume $q = 0$ in suitable coordinates for $D$. We locally (around the exposed point) parametrize the curve $(\partial R)_i$  through the exposed point by $\gamma_i = (\gamma_i^1,\gamma_i^2)$ and assume $\gamma_i^1(0) = q = 0$. Let $\gamma_i^1(t) = a_i^1 t + a_i^2 t^2 + \dots$ be the Taylor expansion in $t$ centered at $0$. For small enough $D$ we can assume that
\[
\sum_{i} \frac{\alpha_i}{\pi_2(x) - \pi_2(p_i)} = \frac{\alpha_{i}}{w} + O(1)
\]
We then obtain for  $t \to 0$ that
\[
\pi_1(\Phi(\gamma_i(t))) = a_n(\gamma_i^2(t)) \cdot (\gamma_i^1(t))^n + \dots = \left(\frac{\alpha_i}{a_i^1}\right)^n \frac{1}{t^n} + \dots
\]
Our calculations show that the assumptions of \cite{woldembed}*{Lemma 1} are satisfied; In particular, for $t$ close to zero, we obtain that each projection $\pi_1(\Phi(\gamma_i(t)))$ is $\mathcal{C}^1$-close to the union of two  different straight rays  $L_1$, $L_2$ going to infinity in opposite directions (one ray for $t\to 0^+$ and another ray for $t\to 0^-$). 
As said above by choosing the coefficients $\alpha_i$ accordingly, we achieve that all those rays for all the exposed points are different. Thus, outside a big disc $\Delta_R \subset X /\!/ \CC^+_{\Theta_1} \cong \CC$, the images $\pi_1(\Phi(\gamma_i(t)))$ of the opened boundary curves together with $\Delta_R$ form a ``Mergelyan set'', i.e., the complement has no bounded components. We are not using  \cite{woldembed}*{Lemma 1} as it stands, only the strategy of its proof. The necessary adjustments for our case are indicated in the proof of Lemma \ref{lem-inductive} below.

In the third step, we consider now the ``opened'' boundary curves  $u_i \colon \RR \to X$. Our aim is to push them to infinity by applying inductively a sequence of automorphisms of $X$ while keeping the interior of $R$ in the domain of convergence of this sequence.

Since $X$ is a Stein manifold, we can exhaust it by $\holo(X)$-convex compacts $\{X_m\}_{m \in \NN}$ such that $X_m \subset X_{m+1}^\circ$ and $\cup_{m \in \NN} X_m = X$. Similarly, we exhaust $R$ by compacts $R_m$.

%By $E_1, \dots, E_r$ we denote the finitely many exceptional fibres of $\pi_1$.

We will apply the following lemma inductively:

\begin{lemma}
\label{lem-inductive}
There exists a holomorphic automorphism $\alpha \colon X \to X$ such that
for any given $\holo(X)$-convex compacts $K \subset L \subset X$ and any $\varepsilon > 0$,
%and given point $x_0 \in X$ that is not in an exceptional fiber,
we obtain that
\begin{enumerate}
\item $\| \alpha - \id \|_K < \varepsilon$ (w.r.t.\ any fixed embedding of $X$ into some $\CC^N$)
\item $\alpha \circ u_i(\RR) \cap L = \emptyset$ for every opened boundary curve $u_i \colon \RR \to X$.
%\item $\alpha \circ u_i (\RR) \cap E_\ell = \emptyset$ for every opened boundary curve $u_i \colon \RR \to X$ and every exceptional fiber $E_\ell$.
%\item In a coordinate neighborhood of $x_0$ given by $(\pi_1, \pi_2)$ the Jacobian of $\alpha$ is equal to $1$.
\end{enumerate}
\end{lemma}

\begin{proof}
We obtain $\alpha$ in the form  $\alpha = f \circ s$ where $s$ is a shear map. This is the so called ``precomposition with a shear'' trick invented by Buzzard and Forstneri\v c in \cite{BF}. First, the automorphism $f$ is   constructed using Anders\'en--Lempert theory as in  \cite{woldembed}*{Lemma 1} to expell a big but still compact part of the opened boundary curve from the compact $L$ in $X$. Note that the proof does not make use of the linear structure of $\CC^2$. Also, their application of Stolzenberg's result \cite{Stolzenberg} on the holomorphic convexity of the union of a holomorphically convex compact and finitely many smooth real arcs  attached to it, is justified in our situation, since it holds in any Stein manifold instead of $\CC^2$, see for example  the points (1) and (2) in the proof of \cite{borderedproper}*{Theorem  5.1}.

In the second step, the precomposing shear map $s$ is constructed as the time-$1$ map of a holomorphic shear $h \Theta_1$: Here, the holomorphic function $h$ which is in the kernel of the LND $\Theta_1$, is obtained by using Mergelyan's theorem. It is used to ensure that no parts from ``far away'' in the opened boundary curves can come back to the compact set $L\subset X$ (recall that Anders\'en-Lempert theory gives only approximation on compacts).  Indeed, we verified in the second step above that the assumptions of \cite{woldembed}*{Lemma 1} are satisfied, which allows the use of Mergelyan's theorem. Note that the corresponding GIT quotient is isomorphic to $\CC$ and the set $\CC \setminus (\Delta_R \cup \pi_1 (u_i))$ has no bounded components. Moreover, we choose $R$ big enough so that the finitely many image points of the exceptional fibers $E_1, \dots, E_r$ of $\pi_1$ are contained in $\Delta_R$. Hence, we can apply Mergelyan's theorem in one dimension. 
%Moreover, $\pi_1 (u_i)$ is not going through the finitely many points corresponding to the finitely many exceptional fibres of $\pi_1$, so $s$ can move all points of $u_i$.
%For details, see Ritter \cite{MR3881472}*{p.~71}.

%By transversality, we can further assume that for the automorphism $f$ constructed using Anders\'en--Lempert theory, we have that $f \circ u_i(\RR) \cap E_\ell = \emptyset$ for all opened boundary curves $u_i$ (real dimension $1$) and all exceptional fibres $E_\ell$ (real dimension $2$). Since $s$ is the flow of a shear of the LND $\theta_1$, it does not change the $\pi_1$-projection of a point and this implies our condition $\alpha \circ u_i (\RR) \cap E_\ell = \emptyset$. 
\end{proof}

The final embedding of $R$ is now obtained by applying Lemma \ref{lem-inductive} for each $K = X_m \cup (\alpha_{m-1} \circ \dots \circ \alpha_1)(R_{j(m)})$, $j(m)\ge m$ and $L = X_{m'}$, inductively in $m \in \NN$ as in \cite{borderedproper}*{Theorem 5.1, p.~111}. Here, $j(m)  \in \NN$ is chosen big enough such that $K$ is polynomially convex, see point (2) in the proof of \cite{borderedproper}*{Theorem 5.1, p.~111}, and where $m' \geq m$ is chosen such that $X_{m'} \supset K$ in each step.

For each $m \in \NN$ we obtain an automorphism $\alpha_m \colon X \to X$ such that $\| \alpha_m - \id \|_{X_m \cup (\alpha_{m-1} \circ \dots \circ \alpha_1)(R_{m(j)})} < 2^{-m}$ as well as $\alpha_m \circ (\alpha_{m-1} \circ \dots \circ \alpha_1 \circ u_i(\RR)) \cap X_{m'} = \emptyset$ for every opened boundary curve $u_i$. %and $\alpha_m \circ (\alpha_{m-1} \circ \dots \circ \alpha_1 \circ \gamma_j(\RR)) \cap E_\ell = \emptyset$ for every opened boundary curve $\gamma_i$ and every exceptional fiber $E_\ell$.

Then the limit $g := \lim_{m \to \infty} \alpha_m \circ \alpha_{m-1} \circ \dots \circ \alpha_1$ converges uniformly on compacts of a Fatou--Bieberbach domain $\Omega := \bigcup_{m \in \NN} (\alpha_m \circ \alpha_{m-1} \circ \dots \circ \alpha_1)^{-1}(X_m)$ and defines a biholomorphic map $g \colon \Omega \to X$ (see Prop. 4.1.1. from  \cite{Forstneric-book}. This domain does not contain the opened boundary curves since they are pushed out of the compact $X_{m'}$ in step $m$ (thus going to infinity). However, $R$ is inside the domain of convergence, since the maps $\alpha_m$ are closer and closer to identity on the exhausting compact subsets $R_{j(m)}\supset R_m$. Thus the restriction $g\vert_R \colon R \to X$ is a proper holomorphic embedding of the open Riemann surface $R$ into $X$. \qedhere

%To see that $g$ is injective, consider its Jacobian: $g$ is either injective or degenerate everywhere. However, the degenerate case can be easily excluded by the condition on the Jacobian in $x_0$.
\end{proof}

\clearpage

\section{Open problems}

Recall that we denote the class of flexible manifolds by $\mathrm{FLEX}$ and the class of manifolds with the algebraic overshear density property by $\mathrm{AOSD}$. For the following problem, let us moreover denote the class of manifolds with the algebraic density property by $\mathrm{AD}$.
We know that $\mathrm{AD} \setminus \mathrm{FLEX} \neq \emptyset$, for e.g.\ $\CC \times \CC^\ast$ is in this set, see \cite{A-Z} for a larger list of examples.
By Proposition \ref{aosd-flex} we know that $\mathrm{AD} \cap \mathrm{FLEX} \subseteq  \mathrm{AOSD}$. However, we do not know whether this inclusion is strict. This motivates the following problem.

\begin{problem} \hfill
\label{problemFLEXAD}
\begin{enumerate}
    \item Is $\mathrm{FLEX} \setminus \mathrm{AD} \neq \emptyset$ among affine-algebraic manifolds?
    \item Is $(\mathrm{AD} \cap \mathrm{FLEX}) \setminus \mathrm{AOSD} \neq \emptyset$?
\end{enumerate}{}

\end{problem}

%Of course one could first try to find the solution to
%
%\begin{problem}
%Which complex-affine surfaces do have the algebraic overshear density property? Of course they need to be smooth Gizatullin surfaces.
%\end{problem}

\begin{definition}
Let $X$ be a complex manifold and $\tau \in \aut{X}$. We call $\tau$ a \emph{generalized translation} if for any $\holo(X)$-convex compact $K \subsetneq X$ there exists an $m \in \NN$ such that for the iterate $\tau^m$ it holds that 
\begin{enumerate}
\item $\tau^m(K) \cap K = \emptyset$ and
\item $\tau^m(K) \cup K$ is $\holo(X)$-convex.
\end{enumerate}
\end{definition}

\begin{problem}
Let $X$ be a complex-affine manifold with the algebraic overshear density property. Does there exist a generalized translation \cite{hypercyclic}*{Def.~1.4} on $X$? This would imply that there exist two automorphisms that generate a dense subset of the identity component of the automorphism group $\autid(X)$.
\end{problem}

\begin{problem}
Let $X$ be a complex-affine manifold with the algebraic overshear density property and with $\dim X \geq 3$. Does there exist a compatible pair of LNDs on $X$?
\end{problem}

The result of Anders\'en and Lempert, answering the question of Rosay and Rudin mentioned in the introduction, says that the group generated by overshears is meagre in $\authol(\CC^n)$. 
However, they use only shears and overshears from formulas \eqref{shear} and
\eqref{overshear}
which we would like to call \emph{coordinate shears} and \emph{coordinate overshears}. Our definition is more general. If one uses all LNDs instead of only coordinate shears, the question becomes more tricky. It is for example not clear, whether the group generated by flows of all overshears for all LNDs can be generated by the overshears for a countable number of LNDs. For $\CC^2$, the group of algebraic automorphisms can be generated
by all shears along the coordinate directions. Hence, it is clear that the problem has a positive answer for $\CC^2$. In $\CC^3$, the algebraic coordinate shears are not sufficient to generate the group of algebraic automorphisms by the result of Shestakov and Umirbaev. However, it is conjectured that a single algebraic automorphism of $\CC^n$ together with the affine group generates the algebraic automorphism group of $\CC^n$ (for $n\ge 4$ see
for example \cite{Kanel} Problem (3) in the last section). 

If this is true, it would most likely give a solution to our next problem in case $X= \CC^n$.
In this context we ask the following question.

\begin{problem}
Let $X$ be a complex-affine manifold with the algebraic overshear density property.
Is the group generated by flows of holomorphic overshear vector fields meagre in $\autid{X}$?
\end{problem}

For a partial derivative $\Theta = \frac{\partial}{\partial z_j}$ in $\CC^n$, it is immediate that $\ker \Theta$ is a ring of functions in $n-1$ variables.
An important ingredient in the above-mentioned proof of Anders{\'e}n and Lempert is to have a good growth estimate (in terms of $d$) for the vector space dimension of $\ker \Theta \cap \{ f \in \CC[z_1, \dots, z_n] \,:\, \deg f \leq d \}$ where $\deg$ is the total degree of $f$. This naturally leads to the following more general question for LNDs.

\begin{problem}
Let $X$ be a smooth affine variety over a field $k$ admitting a $k^\ast$-action. Let $\bigoplus_{j=0}^{\infty} B_j = k[X]$ be the grading induced by this $k^\ast$-action, and denote by $\deg$ the corresponding degree function. Let $\Theta$ be a locally nilpotent derivation on $k[X]$. Does the following hold?
\[
\lim_{d \to \infty} \frac{\dim \{ f \in \ker \Theta \,:\, \deg f \leq d \}}{\dim \{ f \in k[X] \,:\, \deg f \leq d \}} = 0
\]
\end{problem}

In \cite{straight} Kaliman shows that an algebraic isomorphism of two affine subvarietes $X, Y \subset \CC^n$ extends to a holomorphic automorphism of $\CC^n$ if $n > \max(2 \dim X, \dim T X)$. The key ingredient in his proof is to use overshears of a suitable LND. More recently, an improved theorem was also given by Kaliman \cite{MR4098881}.

\begin{problem}
Let $Z$ be an affine-algebraic manifold with the algebraic overshear density property.
Let $X$ and $Y$ be algebraically isomorphic subvarieties of $Z$. Does every algebraic isomorphism $X \to Y$ extend to a holomorphic automorphism $Z \to Z$?
\end{problem}

\begin{problem} Can all  results about embeddings  of open Riemann surfaces into $\CC^2$ from the recent article of Alarc{\'o}n and Forstneri\v{c} \cite{alarconforstneric} be generalized to embeddings into complex-affine surfaces with the algebraic overshear density property?
    
\end{problem}
\section*{Acknowledgements}
The first author would like to thank the Mathematical Institute at the University of Bern and the second author would like to thank the Center of Advanced Mathematical Sciences at the American University of Beirut for hosting them. 

The authors would like to thank the referee for the careful review and their help to improve the exposition.

\section*{Grants}
The research of first author was supported by the European Union (ERC Advanced grant HPDR, 101053085 to Franc Forstneri\v{c}) and grant N1-0237 from ARRS, Republic of Slovenia. The research of the second author was partially supported by Schweizerische Nationalfonds Grant 200021-178730.

\end{document}